\documentclass[a4paper,12pt,leqno]{amsart}
\usepackage{amsmath,amsthm,amssymb,latexsym,epsfig,graphicx,subfigure}
\numberwithin{equation}{section}

\newtheorem{thm}{Theorem}[section]

\newtheorem{lm}[thm]{Lemma}

\theoremstyle{definition}

\theoremstyle{definition}
\newtheorem{rem}[thm]{Remark}

\newcommand{\Rn}{\mathbb{R}^{n}}

\newcommand{\R}{\mathbb{R}}

\newcommand {\grtrsim} {\ {\raise-.5ex\hbox{$\buildrel>\over\sim$}}\ }

\newcommand{\khii}{\text{\lower -.4ex\hbox{$\chi$}}}
\DeclareMathOperator{\spt}{spt}

\newcommand{\restrict}{\begin{picture}(12,12)
                       \put(2,0){\line(1,0){8}}
                       \put(2,0){\line(0,1){8}}
                      \end{picture}}

\begin{document}
\title {Hausdorff dimension of intersections with planes and general sets}
\author{Pertti Mattila}


\begin{abstract} 
We give conditions on a general family $P_{\lambda}:\R^n\to\R^m, \lambda \in \Lambda,$ of orthogonal projections  which guarantee that the Hausdorff dimension formula $\dim A\cap P_{\lambda}^{-1}\{u\}=s-m$  holds generically for measurable sets $A\subset\Rn$ with positive and finite $s$-dimensional Hausdorff measure, $s>m$, and with positive lower density. As an application we prove for measurable sets $A,B\subset\Rn$ with positive $s$- and $t$-dimensional measures,  and with positive lower density that if $s + (n-1)t/n > n$, then   $\dim A\cap (g(B)+z) = s+t - n$ for almost all rotations $g$ and  for positively many $z\in\Rn$.
\end{abstract}

\maketitle

\section{introduction}

Let $P_{\lambda}:\R^n\to\R^m, \lambda \in \Lambda,$ be a family of orthogonal projections and suppose that $\Lambda$ is equipped with a measure $\omega$. If this is the full family of orthogonal projections and a Borel set $A\subset\Rn$ has Hausdorff dimension $\dim A > m$, then according to Marstrand's projection theorem, \cite{M}, the Lebesgue measure $\mathcal L^m(A)>0$ for almost all $\lambda$. Kaufman gave a simple proof for this in \cite{Ka} which shows that for any finite Borel measure $\mu$ with finite energy $I_m(\mu)$ the push-forward $P_{\lambda\sharp}\mu$ is absolutely continuous with density in $L^2(\R^m)$ for almost all $\lambda$. Later this method has been applied to many strict subfamilies of projections by several people, see \cite{F1}, \cite{PS}, \cite{O}, and \cite{M8}, and also Chapters 4, 5 and 18 of \cite{M6}. 

Let $A\subset\Rn$ be measurable with respect to the $s$-dimensional Hausdorff measure $\mathcal H^s$ with $0<\mathcal H^s(A)<\infty$. If $s>m$, then for typical $(n-m)$-planes $V$\  $\dim A\cap V=s-m$ due to results originating in \cite{M}, see also \cite{M5}, Chapter 10, and \cite{M6}, Chapter 6. In Section \ref{dimension} we investigate the following question: Suppose we know for some $s>0$ that $P_{\lambda\sharp}\mu\in L^2(\R^m)$ for almost all $\lambda\in \Lambda$  for all in some sense $s$-dimensional measures $\mu$. Can we then conclude that if $A\subset\Rn$ is  $\mathcal H^s$ measurable with $0<\mathcal H^s(A)<\infty$, then $\dim A\cap P_{\lambda}^{-1}\{u\}= s-m$  holds for positively many, in the sense of Lebesgue measure, $u\in\R^m$ and for almost all $\lambda\in\Lambda$? In Theorem \ref{level} we show that this is true if the $L^2$-boundedness holds in a quantitative sense and if $A$ has positive lower density:
\begin{equation}\label{ld}
\liminf_{r\to 0}(2r)^{-s}\mathcal H^s(A\cap B(x,r)) > 0\ \text{for}\ \mathcal H^s\ \text{almost all}\ x\in A. 
\end{equation}

In Section \ref{applications} we apply this to the Hausdorff dimension of intersections.  We prove that if $A\subset\Rn$ is $\mathcal H^s$ measurable    with $0<\mathcal H^s(A)<\infty$,\ $B\subset\Rn$ is $\mathcal H^t$ measurable with $0<\mathcal H^t(B)<\infty$, and both have positive lower density, and if $s+(n-1)t/n>n$, then for almost all orthogonal transformations $g\in O(n)$,\ $\dim A\cap (g(B)+z) = s+t-n$ for positively many $z\in\Rn$. Earlier this  was proved  in \cite{M3} under the conditions $s+t>n, t>(n+1)/2$, and without any lower density assumptions. I believe that both assumptions $t>(n+1)/2$ and $s+(n-1)t/n>n$ are superfluous, and $s+t>n$ should suffice. Under the condition $s+(n-1)t/n>n$ the weaker inequality $\dim A\cap (g(B)+z) \geq s+(n-1)t/n-n$ holds for general measurable sets with positive and finite measure. This follows combining \eqref{dz} with the results of \cite{M4}.

I believe Theorems \ref{level} and \ref{level1} and the inequalities $\dim A\cap (g(B)+z) \geq s+t-n$ in Theorem \ref{inter} should hold without any lower density assumptions, but the method seems to require it. In general, the opposite inequality can fail very badly, see \cite{F2}.

Hausdorff dimension of plane sections has been studied in \cite{M}, \cite{M1}, \cite{Or} and \cite{MO}, and of general intersections in \cite{K}, \cite{M2}, \cite{M3}, \cite{M4}, \cite{M7}, \cite{EIT} and \cite{DF}. They have also been discussed in the books \cite{M5} and \cite{M6}. 

\section{Preliminaries}

We denote by $\mathcal L^n$ the Lebesgue measure in the Euclidean $n$-space $\Rn, n\geq 2,$ and by $\sigma^{n-1}$ the surface measure on the unit sphere $S^{n-1}$. The closed ball with centre $x\in\Rn$ and radius $r>0$ is denoted by $B(x,r)$ or $B^n(x,r)$. We set $\alpha(n)=\mathcal L^n(B^n(0,1))$. The orthogonal group of $\Rn$ is $O(n)$ and its Haar probability measure is $\theta_n$. For $A\subset\Rn$  we denote by $\mathcal M(A)$ the set of non-zero finite Borel measures $\mu$ on $\Rn$ with compact support $\spt\mu\subset A$. The Fourier transform of $\mu$ is defined by
$$\widehat{\mu}(x)=\int e^{-2\pi ix\cdot y}\,d\mu y,~ x\in\Rn.$$

For $0<s<n$ the $s$-energy of $\mu\in\mathcal M(\Rn)$ is 
\begin{equation}\label{eq10}
I_s(\mu)=\iint|x-y|^{-s}\,d\mu x\,d\mu y=c(n,s)\int|\widehat{\mu}(x)|^2|x|^{s-n}\,dx.
\end{equation} 
The second equality is a consequence of Parseval's formula and the fact that the distributional Fourier transform of the Riesz kernel $k_s, k_s(x)=|x|^{-s}$, is a constant multiple of $k_{n-s}$, see, for example, \cite{M5}, Lemma 12.12, or \cite{M6}, Theorem 3.10. These books contain most of the background material needed in this paper.

Notice that if $\mu$ satisfies the Frostman condition $\mu(B(x,r))\leq r^s$ for all $x\in\Rn, r>0$, then $I_t(\mu)<\infty$ for all $t<s$. We have for any Borel set $A\subset\Rn$ with $\dim A > 0$, cf. Theorem 8.9 in \cite{M5},
\begin{equation}\label{eq3}
\begin{split}
\dim A&=\sup\{s:\exists \mu\in\mathcal M(A)\ \text{such that}\ \mu(B(x,r))\leq r^s\ \text{for all}\ x\in\Rn, r>0\}\\
&=\sup\{s:\exists \mu\in\mathcal M(A)\ \text{such that}\ I_s(\mu)<\infty\}.
\end{split}
\end{equation}

We shall denote by $f_{\#}\mu$ the push-forward of a measure $\mu$  under a map  $f: f_{\#}\mu(A)= \mu(f^{-1}(A))$. The restriction of $\mu$ to a set $A$ is defined by $\mu\restrict A(B)=\mu(A\cap B)$. The notation $\ll$ stands for absolute continuity.

The lower and upper $s$-densities of $A\subset\R^n$ are defined by

$$\theta^s_{\ast}(A,x)=\liminf_{r\to 0}(2r)^{-s}\mathcal H^s(A\cap B(x,r)),\ \theta^{\ast s}(A,x)=\limsup_{r\to 0}(2r)^{-s}\mathcal H^s(A\cap B(x,r)).$$ 
If $\mathcal H^s(A)<\infty$, we have by \cite{M5}, Theorem 6.2,
\begin{equation}\label{dens}
\theta^{\ast s}(A,x)\leq 1\ \text{for}\ \mathcal H^s\ \text{almost all}\ x\in A.
\end{equation}

For $\nu\in\mathcal M(\R^m)$ define the derivative at $u\in \R^m$ by

$$D(\nu,u) = \lim_{\delta\to 0} \alpha(m)^{-1}\delta^{-m}\nu (B(u,\delta)),$$
when the limit exists. It does exist and is finite for $\mathcal L^m$ almost all $u\in\R^m$.

The characteristic function of a set $A$ is $\chi_A$. By the notation $M\lesssim N$ we mean that $M\leq CN$ for some constant $C$. The dependence of $C$ should be clear from the context. The notation $M \approx N$ means that $M\lesssim N$ and $N\lesssim M$. By  $c$ we mean positive constants with obvious dependence on the related parameters. 

\section{Dimension of level sets}\label{dimension}

Let $P_{\lambda}:\Rn\to\R^m, \lambda\in\Lambda,$ be orthogonal projections, where $\Lambda$ is a compact metric space.  Suppose that $\lambda\mapsto P_{\lambda}x$ is  continuous for every $x\in\Rn$. Let also $\omega$ be a finite non-zero Borel measure on $\Lambda$. These assumptions are just to guarantee that the measurability of the various functions appearing later can easily be checked and that the forthcoming applications of Fubini's theorem are legitimate.  Much less would suffice, using, for example, the general results of \cite{MM}.

\begin{thm}\label{level}
Let $s>m$. Suppose that $P_{\lambda\sharp}\mu\ll\mathcal L^m$ for $\omega$ almost all $\lambda\in \Lambda$ and that there exists a positive number $C$ such that

\begin{equation}\label{L2}
\iint D(P_{\lambda\sharp}\mu,u)^2\,d\mathcal L^mu\,d\omega\lambda < C
\end{equation}
whenever $\mu\in\mathcal M(B^n(0,1))$ is such that  $\mu(B(x,r))\leq r^s$ for $x\in \Rn, r>0$. 

If $A\subset\R^n$ is $\mathcal H^s$ measurable, $0<\mathcal H^s(A)<\infty$ and $\theta^s_{\ast}(A,x)>0$ for $\mathcal H^s$ almost all $x\in A$, then for $\mathcal H^s\times\omega$ almost all $(x,\lambda)\in A\times\Lambda$,
\begin{equation}\label{e11}
\dim P_{\lambda}^{-1}\{P_{\lambda}x\}\cap A =s-m,
\end{equation}
and for $\omega$ almost all $\lambda\in \Lambda$,
\begin{equation}\label{e12}
\mathcal L^m(\{u\in\R^m: \dim P_{\lambda}^{-1}\{u\}\cap A = s-m\}) > 0.
\end{equation}
\end{thm}
\begin{proof}
Note first that using \eqref{dens} our assumptions imply that $P_{\lambda\sharp}(\mathcal H^s\restrict A)\ll\mathcal L^m$ for $\omega$ almost all $\lambda\in\Lambda$.

For any $\lambda\in \Lambda$ the inequality $\dim P_{\lambda}^{-1}\{u\}\cap A \leq s-m$ for $\mathcal L^m$ almost all $u\in\R^m$ follows for example from \cite{M5}, Theorem 7.7. This implies $\dim P_{\lambda}^{-1}\{P_{\lambda}x\}\cap A \leq s-m$ for $\mathcal H^s$ almost all $x\in A$ whenever $P_{\lambda\sharp}(\mathcal H^s\restrict A)\ll\mathcal L^m$. Hence we only need to prove the opposite inequalities. 

Define $\mu = 10^{-s}\mathcal H^s \restrict A$. Due to \eqref{dens} we may assume that $\mu(B(x,r))\leq (r/2)^s$ for $x\in \Rn, r>0$, by restricting $\mu$ to a suitable subset of $A$ with large measure; the positive lower density property is inherited by subsets by Corollary 6.3 in \cite{M5}. We may also assume that $A$ is compact, which makes it easier to verify the measurabilities. 

For $\delta>0$ define $\mu^{\delta}\in\mathcal M(\Rn)$ by $\mu^{\delta}(B)=\alpha(n)^{-1}\delta^{-n}\int_B\mu(B(x,\delta))\,dx.$
For $a,x\in\Rn, r>0,$ define $T_{a,r}(x)=(x-a)/r$ and let 
$\mu_{a,r}=r^{-s}T_{a,r\sharp}(\mu\restrict B(a,r))\in\mathcal M(B(0,1))$. Then one easily checks that  $(\mu^{\delta})_{a,r}(B(x,\rho))\leq \rho^s$ for $x\in\Rn$ and $\rho>0$. Hence for all $a\in\Rn, r>0, \delta>0$,
\begin{equation}\label{e4}
\iint D(P_{\lambda\sharp}(\mu^{\delta})_{a,r},u)^2\,d\mathcal L^mu\,d\omega\lambda < C.
\end{equation}

Let $ 0<\epsilon < \delta < r < 1$. In the following estimate observe that for any $a, x\in\Rn, u\in\R^m$, if $|P_{\lambda}(x-a)-ru|\leq \delta$, then
$$\mathcal L^n(\{y\in\Rn: |x-y|\leq 2\delta, |P_{\lambda}(y-a)-ru|\leq \epsilon\})\approx \delta^{n-m}\epsilon^{m}.$$
We obtain
\begin{align*}
&\mu(\{x\in B(a,2r): |P_{\lambda}(x-a)-ru|\leq \delta\})\\
&\approx \delta^{m-n}\epsilon^{-m}\int_{\{x\in B(a,2r): |P_{\lambda}(x-a)-ru|\leq \delta\}}\mathcal L^n(\{y\in\Rn: |x-y|\leq 2\delta, |P_{\lambda}(y-a)-ru|\leq \epsilon\})\,d\mu x\\
&\leq \delta^{m-n}\epsilon^{-m}\int_{\{y\in B(a,4r): |P_{\lambda}(y-a)-ru|\leq \epsilon\}}\mu (B(y,2\delta))\,d \mathcal L^ny\\
&= 2^n\alpha(n)\delta^{m}\epsilon^{-m}\int_{\{y\in B(a,4r): |P_{\lambda}((y-a)/4r)-u/4|\leq \epsilon/4r\}}\,d\mu^{2\delta}y\\
&=2^n\alpha(n)(4r)^s\delta^{m}\epsilon^{-m}P_{\lambda\sharp}(\mu^{2\delta})_{a,4r}(B(u/4,\epsilon/4r)).
\end{align*} 
Thus
\begin{align*}
&\mu(\{x\in B(a,2r): |P_{\lambda}(x-a)-ru|\leq \delta\})\\
&\lesssim r^{s-m}\delta^{m}\liminf_{\epsilon\to 0}(\epsilon/4r)^{-m}P_{\lambda\sharp}(\mu^{2\delta})_{a,4r}(B(u/4,\epsilon/4r))\\
&=\alpha(m)r^{s-m}\delta^{m}D(P_{\lambda\sharp}(\mu^{2\delta})_{a,4r},u/4),
\end{align*}
if the limit exists.
Therefore
\begin{equation}
\begin{split}\label{e9}
C4^m&> \iint D(P_{\lambda\sharp}(\mu^{2\delta})_{a,4r},u/4)^2\,d\mathcal L^mu\,d\omega\lambda\\
&\gtrsim r^{2m-2s}\delta^{-2m}\iint (\mu(\{x\in B(a,2r):|P_{\lambda}(x-a)-ru|\leq \delta\})^2\,d\mathcal L^mu\,d\omega\lambda\\
&= r^{2m-2s}\delta^{-2m}\iint_{B(a,2r)}\int_{B(a,2r)}\\ &\mathcal L^m(\{u:|P_{\lambda}(x-a)-ru|\leq \delta, |P_{\lambda}(y-a)-ru|\leq \delta\})\,d\mu x\,d\mu y\,d\omega\lambda\\
&\gtrsim r^{m-2s}\delta^{-m}\iint_{B(a,2r)} \mu(\{y\in B(a,2r):|P_{\lambda}(y-x)|\leq\delta\})\,d\mu x\,d\omega\lambda\\
&\geq r^{m-2s}r^{t}\iint_{B(a,r)} r^{-t}\delta^{-m} \mu(\{y\in B(x,r):|P_{\lambda}(y-x)|\leq\delta\})\,d\mu x\,d\omega\lambda\\
&= r^{-\eta-s}\iint_{B(a,r)} r^{-t}\delta^{-m} \mu(\{y\in B(x,r):|P_{\lambda}(y-x)|\leq\delta\})\,d\mu x\,d\omega\lambda,\\
\end{split}
\end{equation}
where $0<t<s-m$ and $\eta=s-m-t>0$.

Next we want to show  that for $\omega$ almost all $\lambda\in \Lambda$ and $\mu$ almost all $x\in A$, 
\begin{equation}\label{e8}
\lim_{r\to 0}\liminf_{\delta\to 0}r^{-t}\delta^{-m} \mu(\{y\in B(x,r):|P_{\lambda}(y-x)|\leq\delta\}) = 0.
\end{equation}
Let $B\subset A$ be compact and $b$ and $r_0$ positive numbers such that $\mu(B(a,r))\geq br^s$ for $a\in B$ and $0<r<r_0$. By the assumption on positive lower density we can find them so that $\mu(A\setminus B)$ is arbitrarily small, whence it is enough to show \eqref{e8} for $\mu$ almost all $x\in B$.

For $j=j_0,j_0+1,\dots$, with $2^{-j_0}<r_0$ choose $a_{j,i}\in B, i=1,\dots,k_j,$ such that $B\subset\cup_iB(a_{j,i},2^{-j})$ and the balls $B_{j,i}:=B(a_{j,i},2^{-j}), i=1,\dots,k_j,$ have bounded overlap. Let $0<\delta<2^{-j}$ and set

$$f_{j}(x,\lambda,\delta)=2^{jt}\delta^{-m} \mu(\{y\in B(x,2^{-j}):|P_{\lambda}(y-x)|\leq\delta\}),$$
and
$$f_{j}(x,\lambda)=\liminf_{\delta\to 0}f_{j}(x,\lambda,\delta).$$
Then by \eqref{e9}
$$\iint_{B_{j,i}} f_{j}(x,\lambda,\delta)\,d\mu x\,d\omega\lambda \lesssim 2^{-\eta j-sj}\leq 2^{-\eta j}b^{-1}\mu(B_{j,i}).$$
By the bounded overlap,
$$\iint_{B} f_{j}(x,\lambda,\delta)\,d\mu x\,d\omega\lambda \lesssim  2^{-\eta j}b^{-1}\mu(A).$$
Hence by Fatou's lemma,
$$\iint_{B} f_{j}(x,\lambda)\,d\mu x\,d\omega\lambda \lesssim  2^{-\eta j}b^{-1}\mu(A),$$
whence
$$\iint_{B}\sum_{j\geq j_0} f_{j}(x,\lambda)\,d\mu x\,d\omega\lambda < \infty.$$
Recalling the definition of $f_j$ we have for $\omega$ almost all $\lambda\in \Lambda$ and $\mu$ almost all $x\in B$, 
$$\lim_{j\to\infty}\liminf_{\delta\to 0}2^{jt}\delta^{-m} \mu(\{y\in B(x,2^{-j}):|P_{\lambda}(y-x)|\leq\delta\})=0. $$
This implies \eqref{e8}.

To finish the proof set for $\lambda\in\Lambda$,

$$E_{\lambda}=\{x\in A:\mathcal H^t(P_{\lambda}^{-1}\{P_{\lambda}x\}\cap A)=0\}.$$ 
Then by Lemma \ref{marstrand} below
$$\limsup_{r\to 0}\liminf_{\delta\to 0}r^{-t}\delta^{-m} \mu(\{y\in B(x,r):|P_{\lambda}(y-x)|\leq\delta\}) = \infty$$
for $\mu$ almost all $x\in E_{\lambda}$. On the other hand, by \eqref{e8} for $\omega$ almost all $\lambda\in \Lambda$ and $\mu$ almost all $x\in A$,
$$\limsup_{r\to 0}\liminf_{\delta\to 0}r^{-t}\delta^{-m} \mu(\{y\in B(x,r):|P_{\lambda}(y-x)|\leq\delta\}) = 0.$$
Hence $\mu(E_{\lambda})=0$ for $\omega$ almost all $\lambda\in \Lambda$. It follows from Fubini's theorem 
that $\mu\times\omega(\{(x,\lambda): x\in E_{\lambda}\})=0,$ and so 
$\dim P_{\lambda}^{-1}\{P_{\lambda}x\}\cap A \geq t$ for $\mu\times\omega$ 
almost all $(x,\lambda)\in A\times\Lambda$. Now \eqref{e11} follows by the arbitrariness of $t<s-m$ and then \eqref{e12} follows from the absolute continuity.
\end{proof}

\begin{lm}\label{marstrand}
Let  $t>0$. Suppose that $E\subset\Rn$ is a Borel set  
and $P:\Rn\to\R^m$ is an orthogonal projection. If 
$\mathcal H^t(E\cap P^{-1}\{u\})=0$ for all $u\in\R^m$, then for any $\mu\in\mathcal M(E)$,
$$\limsup_{r\to 0}\liminf_{\delta\to 0}r^{-t}\delta^{-m} \mu(\{y\in E\cap B(x,r):|P(y-x)|\leq\delta\}) = \infty$$
for $\mu$ almost all $x\in E$. 
\end{lm}

Essentially this was proved by Marstrand in \cite{M}, Lemma 16, in the plane. The same proof works here, but I give a partially different argument.

\begin{proof}
Let $F\subset E$ be compact such that for some positive numbers $r_0$ and $C$ we have for $x\in F$ and $0<r<r_0$,
\begin{equation}\label{e2}
\liminf_{\delta\to 0}r^{-t}\delta^{-m} \mu(\{y\in F\cap B(x,r):|P(y-x)|\leq\delta\}) < C.
\end{equation}
It suffices to show that $\mu(F)=0$. 

For fixed $x_0\in F, 0<r<r_0/2$, define $\nu=P_{\sharp}(\mu\restrict F\cap B(x_0,r))\in\mathcal M(P(F\cap B(x_0,r)))$. If $u\in \spt\nu$ then $u=Px$ for some $x\in F\cap B(x_0,r)$. By \eqref{e2},
$$\liminf_{\delta\to 0}\delta^{-m}\nu(B(u,\delta))\leq \liminf_{\delta\to 0}\delta^{-m}\mu(\{y\in B(x,2r):|P(y-x)|\leq\delta\}) <(2r)^t C.$$
This implies that $\nu\ll\mathcal L^m$ and for $\delta>0$,
$$\nu(B(Px_0,\delta))=\int_{B(Px_0,\delta)} D(\nu,u)\,d\mathcal L^mu < (2r)^tC\delta^m,$$
whence
\begin{equation}\label{mar}
\mu(\{y\in F\cap B(x_0,r):|P(y-x_0)|\leq\delta\}) < (2r)^tC\delta^m.
\end{equation}
 
We can find a point $u\in\R^m$ and  $c, \delta_u>0$ such that 
\begin{equation}\label{e10}
\mu(F\cap P^{-1}(B(u,\delta))\geq c\delta^m\mu(F)
\end{equation}  for $0<\delta<\delta_u$. This follows by an easy application of Vitali's covering theorem in $\R^m$. 

Let $\epsilon>0$ and $V= P^{-1}\{u\}.$ As $\mathcal H^t(F\cap V)=0$ and $F\cap V$ is compact, there are balls $B_i=B(x_i,r_i), x_i\in F\cap V, i=1,\dots,k,$ with the balls $B(x_i,r_i/2)$ covering $F\cap V$, such that $r_i<r_0$ and $\sum_{i=1}^kr_i^t<\epsilon$. Notice that unless $\mu(F)=0$, $F\cap V\not=\emptyset$ by \eqref{e10}. For sufficiently small $\delta>0$,
\begin{equation}\label{e15}
F\cap P^{-1}(B(u,\delta)) \subset \bigcup_{i=1}^kF\cap B(x_i,r_i)\cap P^{-1}(B(u,\delta)),
\end{equation}
and by \eqref{mar},
\begin{equation}\label{e3}
\delta^{-m} \mu(F\cap B(x_i,r_i)\cap P^{-1}(B(u,\delta))) < (2r_i)^{t}C.
\end{equation}
Putting together \eqref{e10}, \eqref{e15} and \eqref{e3}, we obtain
$$c\mu(F)\leq 2^tC\sum_{i=1}^kr_i^t<2^tC\epsilon,$$ 
from which the lemma follows.
\end{proof}

\begin{rem}
If $P:\Rn\to\R^m$ is an orthogonal projection, we can take in Theorem \ref{level} $\Lambda = \{P\}$ and $\omega$ the point mass to get a result for an individual projection. However I don't know of any case where this could be useful.
\end{rem}

For an application to intersections we shall need the following product set version of Theorem \ref{level}. There $P_{\lambda}:\Rn\times\R^p\to\R^m, \lambda\in\Lambda,$ are orthogonal projections with the same assumptions as before.

\begin{thm}\label{level1}
Let $s,t>0$ with $s+t>m$. Suppose that $P_{\lambda\sharp}(\mu\times\nu)\ll\mathcal L^m$ for $\omega$ almost all $\lambda\in \Lambda$ and there exists a positive number $C$ such that

\begin{equation}\label{L2}
\iint D(P_{\lambda\sharp}(\mu\times\nu),u)^2\,d\mathcal L^mu\,d\omega\lambda < C
\end{equation}
whenever $\mu\in\mathcal M(B^n(0,1)), \nu\in\mathcal M(B^p(0,1))$ are such that $\mu(B(x,r))\leq r^s$ for $x\in \Rn, r>0$, and $\nu(B(y,r))\leq r^t$ for $y\in \R^p, r>0$. 

If $A\subset\R^n$ is $\mathcal H^s$ measurable, $0<\mathcal H^s(A)<\infty$,\ $B\subset\R^p$ is $\mathcal H^t$ measurable, $0<\mathcal H^t(B)<\infty$, $\theta^s_{\ast}(A,x)>0$ for $\mathcal H^s$ almost all $x\in A$, and $\theta^t_{\ast}(B,y)>0$ for $\mathcal H^t$ almost all $y\in B$, then for $\mathcal H^s\times \mathcal H^t\times\omega$ almost all $(x,y,\lambda)\in A\times B\times\Lambda$,
\begin{equation}\label{pr1}
\dim P_{\lambda}^{-1}\{P_{\lambda}(x,y)\}\cap (A\times B) =s+t-m,
\end{equation}
and for $\omega$ almost all $\lambda\in \Lambda$,
\begin{equation}\label{pr2}
\mathcal L^m(\{u\in\R^m: \dim P_{\lambda}^{-1}\{u\}\cap (A\times B) = s+t-m\}) > 0.
\end{equation}
\end{thm}

\begin{proof}
The proof is essentially the same as that of Theorem \ref{level}. We now have the inequality $\dim P_{\lambda}^{-1}\{u\}\cap (A\times B) \leq \dim A\times B-m$ for almost all $u\in\R^m$ by \cite{M5}, Theorem 7.7, and by \cite{M5},  Theorem 6.13 and Corollary 8.11, the positive lower densities imply $\dim A\times B = s+t$. The corresponding inequality for \eqref{pr1} follows from absolute continuity. So we again only need to prove the opposite inequalities.  For them we just apply the same argument to $\mu\times\nu=(\mathcal H^s\restrict A)\times(\mathcal H^t\restrict B)$ in place of $\mu=\mathcal H^s\restrict A$.
\end{proof}

\section{Intersections}\label{applications}

We now apply Theorem \ref{level1} to the Hausdorff dimension of intersections.

\begin{thm}\label{inter}
Let $s,t>0$ with $s+(n-1)t/n>n$ and let $A\subset\R^n$ be $\mathcal H^s$ measurable, $0<\mathcal H^s(A)<\infty$, and let $B\subset\R^n$ be $\mathcal H^t$ measurable, $0<\mathcal H^t(B)<\infty$,\ $\theta^s_{\ast}(A,x)>0$ for $\mathcal H^s$ almost all $x\in A$, and $\theta^t_{\ast}(B,y)>0$ for $\mathcal H^t$ almost all $y\in B$. Then for $\mathcal H^s\times\mathcal H^t\times\theta_n$ almost all $(x,y,g)\in A\times B\times O(n)$,
\begin{equation}\label{e5}
\dim  A\cap (g(B-y)+x) = s+t-n,
\end{equation}
and for $\theta_n$ almost all $g\in O(n)$,
\begin{equation}\label{e6}
\mathcal L^n(\{z\in\R^n:\dim A\cap (g(B)+z) = s+t-n\}) > 0.
\end{equation}
\end{thm}

\begin{proof}

We apply Theorem \ref{level1} with $P_g:\R^{2n}\to\R^n, P_g(x,y)=x-g(y), x,y\in\Rn, g\in O(n)$. The validity of its assumptions follows from the proof of Theorem 4.2 in \cite{M8}, but I give the short argument here. It is based on the estimates of Wolff \cite{W} and Du and Zhang \cite{DZ} on quadratic spherical averages of the Fourier transform.

Let $\mu, \nu\in\mathcal M(B^n(0,1)),$ with $\mu(B(x,r))\leq r^{s}, \nu(B(x,r))\leq r^{t}$ for $x\in\R^n, r>0$.  Set for $r>1$,

$$\sigma(\nu)(r)=\int_{S^{n-1}}|\widehat{\nu}(rv)|^2\,d\sigma^{n-1}v.$$ 
Let $0<t'<t$ with $s+(n-1)t'/n>n$. Then by \cite{DZ},
\begin{equation}\label{dz}
\sigma(\nu)(r) \lesssim r^{-(n-1)t'/n}.
\end{equation} 
To apply Theorem \ref{level1} we need that the implicit constant here is independent of $\nu$ as long as $\nu\in\mathcal M(B^n(0,1))$ and $\nu(B(x,r))\leq r^{t}$ for $x\in\R^n, r>0$. 
It is not stated in \cite{DZ}, but it can be checked from the proofs.

As   $\widehat{P_{g\sharp}(\mu\times\nu)}(\xi)=\widehat{\mu}(\xi)\widehat{\nu}(-g^{-1}(\xi))$ we have 
\begin{equation}\label{e7}
\begin{split}
&\iint|\widehat{P_{g\sharp}(\mu\times\nu)}(\xi)|^2\,d\xi\,d\theta_n g
=c\int\sigma(\nu)(|\xi|)|\widehat{\mu}(\xi)|^2\,d\xi\\
&\lesssim \mathcal L^n(B(0,1)) + \int_{|\xi|>1}|\widehat{\mu}(\xi)|^2|\xi|^{-(n-1)t'/n}\,d\xi\\
&=\mathcal L^n(B(0,1)) +c'I_{n-(n-1)t'/n}(\mu)\leq  C(n,s,t')<\infty.
\end{split}
\end{equation}
since $n-(n-1)t'/n<s$.  

We can now apply \eqref{pr1}
 of Theorem \ref{level1}. It gives $\dim P_g^{-1}\{P_g(x,y)\}\cap (A\times B) =s+t-n$ for $\mathcal H^s\times\mathcal H^t\times\theta_n$ almost all $(x,y,g)\in A\times B\times O(n)$. Notice that $(u,v)\in P_g^{-1}\{P_g(x,y)\}\cap (A\times B)$ if and only if $u\in A, v\in B$ and $u=g(v-y)+x$, that is,
$$A\cap (g(B-y)+x)=\Pi(P_g^{-1}\{P_g(x,y)\}\cap (A\times B)),$$
where the projection $\Pi(x,y)=x$ is a constant times isometry on any $n$-plane $\{(u,v):u=g(v)+w\}$.
Hence \eqref{e5} follows. In the same way \eqref{e6} follows from \eqref{pr2} of Theorem \ref{level1}.
\end{proof}

In \cite{M8} \eqref{e7} was proven also for other measures on $O(n)$ in place of $\theta_n$ yielding dimension estimates for exceptional subsets of $O(n)$. Combining this with Theorem \ref{level1} we obtain with the same proof as above:

\begin{thm}
Suppose the assumptions of Theorem \ref{inter} are valid. Let $E$ be the set of $g\in O(n)$ for which one of the conclusions \eqref{e5} or \eqref{e6} fails. Then $\dim E \leq 2n-1-s - (n-1)t/n +(n-1)(n-2)/2.$
\end{thm}

The assumptions of Theorem \ref{level} are known to hold in many cases, consequently we obtain dimension formulas for the corresponding plane sections.  For example, $(x,y)\mapsto x-ty, x,y\in\Rn, t\in\R,$ is a special case of projections considered by Oberlin in \cite{O} and $(x,y)\mapsto x-g(y), x,y\in\Rn, g\in O(n),$ was studied in \cite{M8}. However, it seems that Orponen's methods from \cite{Or} yield the same results and without any lower density assumptions.

\vspace{1cm}
\begin{footnotesize}
{\sc Department of Mathematics and Statistics,
P.O. Box 68,  FI-00014 University of Helsinki, Finland,}\\
\emph{E-mail address:} 
\verb"pertti.mattila@helsinki.fi" 

\end{footnotesize}


\begin{thebibliography}{CMM}




\bibitem[DF]{DF} C. Donoven and K.J. Falconer. Codimension formulae for the intersection of fractal subsets of Cantor spaces,   {\em Proc. Amer. Math. Soc.} {\bf 144} (2016), 651--663. 


\bibitem[DZ]{DZ} X. Du and R. Zhang. Sharp $L^2$ estimates of the Schr\"odinger maximal function in higher
dimensions,  {\em Annals of Math.} {\bf 189} (2019), 837--861.

\bibitem[EIT]{EIT} S. Eswarathasan, A. Iosevich and K. Taylor. Fourier integral operators, fractal sets, and the regular values theorem,  {\em Adv. Math.} {\bf 228} (2019), 2385--2402.



\bibitem[F1]{F1}
K.J. Falconer. Hausdorff dimension and the exceptional set of projections, 
{\em Mathematika} {\bf 29} (1982), 109--115.

\bibitem[F2]{F2}
K.J. Falconer. Classes of sets with large intersection, 
{\em Mathematika} {\bf 32} (1985), 191--205.











\bibitem[K]{K}
J.--P. Kahane. 
Sur la dimension des intersections, In Aspects of Mathematics and Applications, North-Holland Math. Library, 34, (1986), 419--430.

\bibitem[Ka]{Ka}
R. Kaufman.
On Hausdorff dimension of projections,
{\em Mathematika} {\bf 15} (1968), 153-155.


 
 
 
 
 
 \bibitem[M]{M}
J.~M. Marstrand.
 Some fundamental geometrical properties of plane sets of fractional
  dimensions, {\em Proc. London Math. Soc.(3)} {\bf 4} (1954), 257--302.

\bibitem[M1]{M1}
P.~Mattila.
 Hausdorff dimension, orthogonal projections and intersections with
  planes,
 {\em Ann. Acad. Sci. Fenn. A Math.} {\bf 1} (1975),  227--244.
 
\bibitem[M2]{M2} P. Mattila. Hausdorff dimension and capacities of intersections of sets in n-space, {\em Acta Math.}  {\bf 152}, (1984), 77--105.

\bibitem[M3]{M3} P. Mattila. On the Hausdorff dimension and capacities of intersections, {\em Mathematika}  {\bf 32}, (1985), 213--
217.

\bibitem[M4]{M4} P. Mattila. Spherical averages of Fourier transforms of measures with finite energy; dimension of intersections and distance sets, {\em Mathematika}  {\bf 34}, (1987), 207--228.


\bibitem[M5]{M5}
P. Mattila.
{\em Geometry of Sets and Measures in Euclidean Spaces},
Cambridge University Press, Cambridge, 1995.


\bibitem[M6]{M6}
P. Mattila.
{\em Fourier Analysis and Hausdorff Dimension},
Cambridge University Press, Cambridge, 2015.

\bibitem[M7]{M7} P. Mattila. Exceptional set estimates for the Hausdorff dimension of intersections, {\em Ann. Acad. Sci. Fenn. A Math.} {\bf 42} (2017),  611--620.

\bibitem[M8]{M8} P. Mattila. Hausdorff dimension and projections related to intersections, arXiv:2005.04947.

\bibitem[MM]{MM}
P. Mattila and R. D. Mauldin. Measure and dimension functions: measurability and densities, {\em Math. Proc. Cambridge Phil. Soc.} {\bf 121} (1997), 81--100.

\bibitem[MO]{MO} P. Mattila and T. Orponen. Hausdorff dimension, intersection of projections and exceptional plane sections,   {\em Proc. Amer. Math. Soc.} {\bf 144} (2016), 3419--3430.

\bibitem[O]{O} D. M. Oberlin. Exceptional sets of projections, unions of k-planes, and associated transforms,  {\it Israel J. Math.} {\bf 202} (2014),  331--342.


\bibitem[Or]{Or} T. Orponen. Slicing sets and measures, and the dimension of exceptional parameters,   {\em J. Geom. Anal.} {\bf 24} (2014),  47--80.



\bibitem[PS]{PS} Y. Peres and W. Schlag. Smoothness of projections, Bernoulli convolutions, and the dimension 
of exceptions, {\em Duke Math. J.} {\bf 102} (2000), 193--251.   


\bibitem[W]{W} T. W. Wolff. Decay of circular means of Fourier transforms of measures,  {\em Int. Math. Res. Not.}  {\bf 10} (1999), 547--567.
	
\end{thebibliography}
\end{document}